\documentclass[a4paper,12pt]{amsart}

\usepackage{amscd}

\usepackage{amsfonts}
\usepackage{amsmath}
\usepackage{amssymb}
\usepackage{latexsym}
\usepackage{amsthm,color}

\usepackage[all]{xy}

\usepackage{graphicx}
\usepackage{color}

\newtheorem{proposition}{Proposition}
\newtheorem{lemma}{Lemma}[section]
\newtheorem{sublemma}{Sublemma}[section]
\newtheorem{theorem}{Theorem}
\newtheorem{corollary}{Corollary}

\theoremstyle{definition}
\newtheorem{definition}{Definition}[section]

\theoremstyle{definition}

\theoremstyle{definition}
\newtheorem{remark}{Remark}[section]

\newcommand{\R}{\mathbb{R}}

\begin{document}
\title[
Simultaneous {smoothness and simultaneous} stability
]
{Simultaneous {smoothness and simultaneous} stability 
\\ 
of a $C^\infty$ {strictly} convex integrand and its dual}
\author[E.B.~Batista]{Erica Boizan Batista}
\address{
{\color{black}Center of Science and Technology, Federal
University of Cariri, 63048-080, Juazeiro do Norte -- CE, Brazil.
}
}
\email{{\color{black} erica.batista@ufca.edu.br}}
\author[H.~Han]{Huhe Han}
\address{
{\color{black}Research Institute of Environment and Information Sciences},  
Yokohama National University, 
Yokohama 240-8501, Japan
}
\email{han-huhe-bx@ynu.jp}
\author[T.~Nishimura]{Takashi Nishimura
}
\address{
Research Institute of Environment and Information Sciences,  
Yokohama National University, 
Yokohama 240-8501, Japan}
\email{nishimura-takashi-yx@ynu.jp}

\begin{abstract}
{
\color{black}In this paper, 
we investigate   
simultaneous properties of a convex integrand $\gamma$ and its dual $\delta$.   The main results are the following three.   
}
\begin{enumerate} 
\item[(1)]\quad 
For a $C^\infty$ convex integrand $\gamma: S^n\to \mathbb{R}_+$, 
its dual convex integrand $\delta: S^n\to \mathbb{R}_+$ 
is of class $C^\infty$ {\color{black}if and only if $\gamma$ is a strictly 
convex integrand}.   
\item[(2)]\quad 
{\color{black}Let $\gamma: S^n\to \mathbb{R}_+$ be a 
$C^\infty$ strictly convex integrand. Then,     
$\gamma$ is stable if and only if}   
its dual convex integrand $\delta: S^n\to \mathbb{R}_+$ 
is stable.   
\item[(3)]\quad 
{\color{black}Let $\gamma: S^n\to \mathbb{R}_+$ be a 
$C^\infty$ strictly convex integrand. 
Suppose that $\gamma$ is stable.}    
Then, for any $i$ $(0\le i\le n)$, 
{\color{black}a point} 
$\theta_0\in S^n$ is a non-degenerate critical point of $\gamma$ 
with Morse index $i$ if and only if 
{\color{black}its antipodal point} $-\theta_0\in S^n$ is 
a non-degenerate critical point of the dual convex integrand ${\delta}$ 
with Morse index $(n-i)$.     
\end{enumerate}
  \end{abstract}





\date{}

\maketitle

\section{Introduction}
Throughout this paper, we let $n$ and $\mathbb{R}_+$ be a positive integer 
and the set consisting of positive real numbers 
respectively.     
Let inv$:\R^{n+1}-\{{\bf 0}\}\to\R^{n+1}-\{{\bf 0}\}$ be the inversion 
with respect to the origin ${\bf 0}$ of $\R^{n+1}$, namely, 
inv$:\R^{n+1}-\{{\bf 0}\}\to\R^{n+1}-\{{\bf 0}\}$ is defined as follows
 where $(\theta, r)$ means the polar plot expression 
for a point of $\R^{n+1}-\{{\bf 0}\}$:
\[
\mbox{\rm inv}(\theta,r)=\left(-\theta,\dfrac{1}{r}\right).    
\]
Let $S^n$ be the unit sphere of $\mathbb{R}^{n+1}$.    
For a continuous function 
$\gamma: S^n\to \mathbb{R}_+$,    
denote the boundary of 
the convex hull of inv(graph$(\gamma)$) by $\Gamma_\gamma$, 
where graph$(\gamma)$ is the subset of $\mathbb{R}^{n+1}-\{{\bf 0}\}$ 
defined as follows.  
\[
\mbox{\rm graph}(\gamma)=
\left\{(\theta,\gamma(\theta))\in \mathbb{R}^{n+1}-\{{\bf 0}\}\; 
\left.\right|\; \theta\in S^n\right\}.  
\]
A continuous function $\gamma: S^n\to \mathbb{R}_+$ is said to be a  
{\it convex integrand} if the equality 
$\Gamma_\gamma=$inv(graph$(\gamma)$) is satisfied 
{\color{black}(\cite{taylor})}.   
{\color{black}
\begin{definition}\label{strictly convex integrand}
{\rm 
A convex integrand $\gamma: S^n\to \mathbb{R}_+$ is 
{\color{black}called} a {\it strictly convex integrand} if 
the convex hull of inv(graph$(\gamma)$) is strictly convex.    
}
\end{definition}
}
Define $C^\infty(S^n, \mathbb{R}_+)$ and 
$C^\infty_{\rm conv}(S^n, \mathbb{R}_+)$ as follows.    
\begin{eqnarray*}
C^\infty(S^n, \mathbb{R}_+) & = & 
\left\{\gamma: S^n\to \mathbb{R}_+ \;\; C^\infty\right\},  \\ 
C^\infty_{\rm conv}(S^n, \mathbb{R}_+) & = & 
\left\{\gamma\in C^\infty(S^n, \mathbb{R}_+)\; |\; \gamma 
\mbox{ is a convex integrand}\right\}.   
\end{eqnarray*}  
The set $C^\infty(S^n, \mathbb{R}_+)$ is a topological space endowed with 
Whitney $C^\infty$ topology 
(for details on Whitney $C^\infty$ topology, for instance 
see \cite{arnold-guseinzade-varchenko, brocker, g-g, narasimhan}); and 
the set $C^\infty_{\rm conv}(S^n, \mathbb{R}_+)$ 
is a topological subspace of 
$C^\infty(S^n, \mathbb{R}_+)$.     
Given a $\gamma\in C^\infty_{{\rm conv}}(S^n, \mathbb{R}_+)$, 
the {\it Wulff shape associated with $\gamma$}, 
denoted by $\mathcal{W}_\gamma$, is the following intersection: 
\[
\bigcap_{\theta\in S^n}\left\{
\left.x\in \mathbb{R}^{n+1}\; \right|\; x\cdot \theta\le \gamma(\theta)  
\right\}, 
\]       
where $x\cdot \theta$ stands for 
the standard scalar product of two vectors $x$ and 
$\theta$ of $\mathbb{R}^{n+1}$.    
The notion of Wulff shape was firstly introduced 
by G. Wulff \cite{wulff} in 1901 
as a geometric model of a crystal at equilibrium.     
For details on Wulff shapes, see for instance 
\cite{giga, morgan, crystalbook, taylor, taylor2}.   
By definition, {\color{black}any} Wulff shape $\mathcal{W}_\gamma$ is compact, 
convex and contains the origin of $\R^{n+1}$ as an interior point.     
{\color{black}
In order to investigate Wulff shapes, the notion of convex integrand was introduced (see \cite{taylor}).    
}   
{\color{black} 
\begin{definition}\label{definition 1.2}   
{\rm 
Let $\gamma: S^n\to \mathbb{R}_+$ be a convex integrand.   
\begin{enumerate}
\item[(1)]\quad A convex integrand 
$\delta: S^n\to \mathbb{R}_+$ is called 
the {\it dual convex integrand of $\gamma$} or just the {\it dual} 
of $\gamma$ 
if the equality 
inv(graph$(\delta))=\partial \mathcal{W}_\gamma$ holds, 
where $\partial \mathcal{W}_\gamma$ stands for the boundary of 
$\mathcal{W}_\gamma$.     
\item[(2)]\quad The Wulff shape associated with $\delta$ 
is called the {\it dual Wulff shape of} $\mathcal{W}_\gamma$ 
and is denoted by $\mathcal{DW}_\gamma$.  
\[
\mathcal{DW}_\gamma=\mathcal{W}_\delta. 
\]
\end{enumerate}
}
\end{definition}
}
\noindent 
Notice that both of the above two dual notions are involutive, namely, 
we have that the dual convex integrand of $\delta$ is $\gamma$ and 
the equality $\mathcal{DDW}_\gamma=\mathcal{W}_\gamma$ 
holds {\color{black}
(see Lemma \ref{lemma 2.7} in Subsection \ref{subsection 2.6})}.     
{
\color{black}In this paper, it is focused exclusively on investigating  
simultaneous properties of a convex integrand 
$\gamma$ and its dual $\delta$. 
}  
\par 
\medskip 
{\color{black}
\begin{definition}\label{stable}
{\rm 
A $C^\infty$ function $\gamma\in C^\infty(S^n, \mathbb{R}_+)$ 
is said to be 
{\it stable} if the $\mathcal{A}$-equivalence class of $\gamma$ is open, 
where 
two elements $\gamma_1, \gamma_2\in C^\infty(S^n, \mathbb{R}_+)$ are 
said to be $\mathcal{A}$-{\it equivalent} 
if there exist $C^\infty$ diffeomorphisms 
$h: S^n\to S^n$ and $H: \mathbb{R}_+\to \mathbb{R}_+$ satisfying  
$\gamma_1=H\circ \gamma_2\circ h^{-1}$.   
}
\end{definition}
}
\noindent 
{\color{black} 
There are two reasons why we 
prefer stable convex integrands.        
It is well-known that a convex integrand represents surface energy density.     
And, usually, it is almost impossible to obtain the 
precise surface energy density, that is to say, 
we merely have an approximated surface energy density.   
Therefore, we would like to have a situation that the space consisting of 
$C^\infty$ convex integrands having only non-degenerate critical points 
is dense in the space consisting of $C^\infty$ convex integrands.   
This is one reason.      
Another reason is as follows.  
The Morse inequalities (\cite{morse theory}) are  
almost indispensable tools for investigating 
convex integrands globally from the 
differentiable viewpoint.     
In order to apply the Morse inequalities, critical points must be 
non-degenerate.        
}
\par 
In \cite{batistahannishimura}, it 
{
\color{black}
has been shown that 
}
stable convex integrands form an open and dense subset of 
$C^\infty_{\rm conv}(S^n, \mathbb{R}_+)$.       
As the next step of \cite{batistahannishimura}, 
it is natural to investigate {\color{black}when and} how 
{\color{black}the simultaneous stability of} the convex integrand 
$\gamma$ {\color{black}and its dual $\delta$ occurs}, 
which is the main purpose of this paper.   
{\color{black}
Since stable functions must be of class $C^\infty$, 
before carrying out the main purpose, it is necessary to 
show the following:}    
\begin{theorem}\label{theorem 1}
Let $\gamma: S^n\to \mathbb{R}_+$ be a $C^\infty$ convex integrand 
and let $\delta$ be the dual convex integrand of $\gamma$.     
Then, $\delta$ is of class $C^\infty$ {\color{black}if and only if 
$\gamma$ is a strictly convex integrand}.   
\end{theorem} 
\noindent 
{\color{black}It should be noted that weaker versions of} 
Theorem \ref{theorem 1} ha{\color{black}ve} been 
already shown independently 
in \cite{andrews, morgan, soner}.    
{\color{black}
Notice also that by \cite{hannishimura2}, 
the assumption of Theorem \ref{theorem 1} 
implies that $\delta$ is a strictly convex integrand.   
Thus, we have the following corollary.   
\begin{corollary}\label{corollary 0}
Let $\gamma: S^n\to \mathbb{R}_+$ be a convex integrand 
and let $\delta$ be its dual convex integrand.  
Then, $\gamma$ is a $C^\infty$ strictly convex integrand 
if and only if $\delta$ is a $C^\infty$ strictly convex integrand.   
\end{corollary}     
}
\noindent 
{\color{black}Define} the functions  
$\widehat{\gamma}, \widehat{\delta} : S^n\to \mathbb{R}_+$ by 
\[
\widehat{\gamma}(\theta)=
\frac{1}{\gamma(-\theta)} 
\;\mbox{ and }\; 
\widehat{\delta}(\theta)=
\frac{1}{\delta(-\theta)}      
\;\quad (\forall \theta\in S^n)
\]
respectively.   
{\color{black}It is easily seen} that 
$\partial \mathcal{DW}_\gamma$ 
{\color{black}(resp., 
$\partial \mathcal{DW}_\delta$)} is the graph of the function 
$\widehat{\gamma}$ {\color{black}(resp., $\widehat{\delta}$) 
and that $\gamma$ (resp., $\delta$) 
is $\mathcal{A}$-equivalent to $\hat{\gamma}$ (resp., $\hat{\delta}$). }        
Thus, as {\color{black}another} corollary of Theorem \ref{theorem 1}, 
we have the following:   
\begin{corollary}\label{corollary 1}
Let $\gamma: S^n\to \mathbb{R}_+$ be a {\color{black}strictly} 
convex integrand and 
let $\delta: S^n\to \mathbb{R}_+$ be 
the dual convex integrand of $\gamma$.   
Then, the following are equivalent.   
\begin{enumerate}
\item[(1)]\quad 
The convex integrand $\gamma$ is of class $C^\infty$. 
\item[(2)]\quad 
The convex integrand $\delta$ is of class $C^\infty$. 
\item[(3)]\quad 
The function $\widehat{\gamma}$, whose graph is exactly  
$\partial \mathcal{W}_\delta=\partial \mathcal{DW}_\gamma$, is 
of class $C^\infty$.   
\item[(4)]\quad 
The function $\widehat{\delta}$, whose graph is exactly  
$\partial \mathcal{W}_\gamma=\partial \mathcal{DW}_\delta$, is of class $C^\infty$.     
\end{enumerate}
\end{corollary}
\par 
Notice that for any convex integrand 
$\gamma: S^n\to \mathbb{R}_+$, the Wulff shape 
$\mathcal{W}_\gamma$ is a convex body 
such that the origin is contained in it's interior.   
In Convex Body Theory, there is the notion of {\it dual} for a 
convex body containing the origin as an interior point.   
Namely, in Convex Body Theory, 
the boundary of the {\it dual} of $\mathcal{W}_\gamma$ is the following 
set (see for example \cite{schneider}).  
\[
\left\{\left.\left(\theta, \frac{1}{\gamma(\theta)}\right)\; \right|\; 
\theta\in S^n\right\}.     
\]
However, 
the notion of dual in this sense seems to have less relations  
with the notion of pedal which seems to be a common background 
in Physics 
(for instance, see \cite{crystalbook}).   
On the other hand, 
the notion of dual Wulff shapes in our sense 
is closely related to the notion of pedal.   
Moreover, via the central projection, 
the pedal of $C^\infty$ embedding 
$\Phi: S^n\to \mathbb{R}^{n+1}-\{{\bf 0}\}$ defined by 
$\Phi(\theta)=\left(\theta, 1/\delta(-\theta)\right)$ 
relative to the origin is characterized 
by using the spherical dual of the corresponding embedding 
$\widetilde{\Phi}: S^n\to S^{n+1}$; 
and the spherical dual is a well-known 
notion in Singularity Theory 
(for details, see Subsection \ref{subsection 2.3}).         
Since pedals are useful to study a Wulff shape associate with 
a $C^\infty$ convex integrand,   
%
we adopt $\mathcal{DW}_\gamma$ as the notion of dual Wulff shape of 
$\mathcal{W}_\gamma$.     

%
%
\par 
\medskip 
{\color{black}The following Theorem \ref{theorem 2} answers 
the question \lq\lq When does the simultaneous stability of 
$\gamma$ and $\delta$ occur ?\rq\rq.}
\begin{theorem}\label{theorem 2}
Let $\gamma: S^n\to \mathbb{R}_+$ be a 
{\color{black} $C^\infty$ strictly} convex integrand 
and let $\delta$ be the dual convex integrand of $\gamma$.   
Then, {\color{black}$\gamma$ is stable if and only if $\delta$ is stable}.   
\end{theorem}
\begin{corollary}\label{corollary 2}
Let $\gamma: S^n\to \mathbb{R}_+$ be a $C^\infty$ 
{\color{black}strictly} convex integrand 
and 
let $\delta: S^n\to \mathbb{R}_+$ be 
the dual convex integrand of $\gamma$.   
Then, the following are equivalent.   
\begin{enumerate}
\item[(1)]\quad 
The convex integrand $\gamma$ is stable. 
\item[(2)]\quad 
The convex integrand $\delta$ is stable. 
\item[(3)]\quad 
The function $\widehat{\gamma}$, whose graph is exactly  
$\partial \mathcal{W}_\delta=\partial \mathcal{DW}_\gamma$, is stable.   
\item[(4)]\quad 
The function $\widehat{\delta}$, whose graph is exactly  
$\partial \mathcal{W}_\gamma=\partial \mathcal{DW}_\delta$, is stable.     
\end{enumerate}
\end{corollary}

\medskip 
\par 
{\color{black}The following Theorem \ref{theorem 3} answers 
the question \lq\lq How does the simultaneous stability of 
$\gamma$ and $\delta$ occur ?\rq\rq.}
\begin{theorem}\label{theorem 3}
Let $\gamma: S^n\to \mathbb{R}_+$ be a 
{\color{black}$C^\infty$ strictly} convex integrand and 
let $\delta: S^n\to \mathbb{R}_+$ be 
the dual convex integrand of $\gamma$.   
{\color{black}Suppose that $\gamma$ is stable.}    
Then, the following hold:   
\begin{enumerate}
\item[(1)]\quad 
A point $\theta_0\in S^n$ 
is a non-degenerate critical point of $\gamma$ if and only if 
{\color{black}its antipodal point} 
$-\theta_0\in S^n$ is a non-degenerate critical point of $\delta$.   
\item[(2)]\quad 
Suppose that {\color{black}a point} $\theta_0\in S^n$ is 
a non-degenerate critical point of $\gamma$.    
Then, 
the Morse index of $\gamma$ at $\theta_0$ is $i$ if and only if 
the Morse index of $\delta$ at $-\theta_0$ is $(n-i)$, 
{\color{black}where $i$ is an integer such that $0\le i\le n$}.         
\end{enumerate}
\end{theorem}
\noindent 
It is clear that, by Theorem \ref{theorem 3}, we have the following corollary.   
\begin{corollary}\label{corollary 3}
Let $\gamma: S^n\to \mathbb{R}_+$ be a stable convex integrand and 
let $\delta: S^n\to \mathbb{R}_+$ be 
the dual convex integrand of $\gamma$.   
Moreover, let $\theta_0$ be a point of $S^n$ 
and let $i$ be an integer such that $0\le i\le n$.   
Then, the following are equivalent.   
\begin{enumerate}
\item[(1)]\quad 
The point $\theta_0\in S^n$ is a non-degenerate 
critical point of $\gamma$ with Morse index $i$.   
\item[(2)]\quad 
The point $-\theta_0\in S^n$ is a non-degenerate 
critical point of $\delta$ with Morse index $(n-i)$.   
\item[(3)]\quad 
The point $-\theta_0\in S^n$ is a non-degenerate 
critical point of $\widehat{\gamma}$ with Morse index $(n-i)$.   
\item[(4)]\quad 
The point $\theta_0\in S^n$ is a non-degenerate 
critical point of $\widehat{\delta}$ 
with Morse index $i$.   
\end{enumerate}
\end{corollary}
\bigskip 
This paper is organized as follows.   
In Section \ref{section 2}, preliminaries 
are given.    
Theorems \ref{theorem 1}, \ref{theorem 2} and \ref{theorem 3} are 
proved in Sections \ref{section 3}, \ref{section 4} 
and \ref{section 5}  respectively.   
\section{Preliminaries}\label{section 2}
\subsection{Stable functions $S^n\to \mathbb{R}_+$}\label{subsection 2.1}
In this subsection, we quickly review 
a geometric characterization of a stable function 
$\gamma: S^n\to \mathbb{R}_+$ 
and the definition of Morse index of $\gamma$ 
at a non-degenerate critical point $\theta\in S^n$.       
Both are well-known.  
\par 
{\color{black}Among} Mather's celebrated series 
\cite{mather1, mather2, mather3, 
mather4, mather5, mather6}, the geometric characterization of a proper 
stable mapping is dealt with in \cite{mather5}.     
It is easily 
seen that the following well-known geometric characterization 
of a stable function $S^n\to \mathbb{R}_+$ is derived 
from Mather's geometric characterization.  
\begin{proposition}[\cite{mather5}]\label{proposition 2.1.1}
A $C^\infty$ function $\gamma: S^n\to \mathbb{R}_+$ is stable 
if and only if 
 all critical points of $\gamma$ are non-degenerate and  
$\gamma(\theta_1)\ne \gamma(\theta_2)$ 
holds 
for any two distinct critical points 
$\theta_1, \theta_2\in S^n$.  
\end{proposition}
\noindent 
It seems that a proper stable function is usually called 
a {\it Morse function}.   
However, a Morse function in \cite{morse theory} is 
a $C^\infty$ function having only 
non-degenerate critical points, 
and thus it is a weaker notion than the notion of stable function.  
Therefore, in order to avoid unnecessary  confusion, 
stable functions and Morse functions are distinguished in this paper.  
\begin{definition}[\cite{morse theory}]\label{definition 2.1}
{\rm Let $\gamma: S^n\to \mathbb{R}_+$ be a 
$C^\infty$ function and let $\theta\in S^n$ be a non-degenerate 
critical point of $\gamma$.    Then,  
there exists a coordinate neighborhood 
$(U, \varphi)$ of $\theta$ 
such that $\varphi(\theta)={\bf 0}$ and the following equality holds:   
\[
\gamma\circ \varphi^{-1}(x_1, \ldots, x_n)= 
\gamma(\theta)-x_1^2-\cdots -x_i^2+x_{i+1}^2+\cdots +x_n^2.           
\]
The integer $i$ $(0\le i\le n)$  
does not depend on the particular choice of 
the coordinate neighborhood $(U, \varphi)$ and it is called 
the {\it Morse index of $\gamma$ at $\theta$}.  
Here, the integer $i$ is more than or equal to $0$ 
and less than or equal to $n$.   
}
\end{definition}
\subsection{Pedals}\label{subsection 2.2}
Although it has been explained only for plane pedal curves in it, 
the reference \cite{bruce-giblin2} is an excellent book for pedals.     
The definition of higher dimensional pedal is parallel to 
the definition of plane pedal curve as follows.   
\begin{definition}\label{definition 2.2}
{\rm 
Given a $C^\infty$ embedding 
$\Phi: S^n\to \mathbb{R}^{n+1}-\{{\bf 0}\}$,  
the {\it pedal } relative to the {\it pedal point} ${\bf 0}$ for $\Phi$, 
denoted by $ped_{\Phi, {\bf 0}}: S^n\to \mathbb{R}^{n+1}$,  
is the mapping which maps $\theta\in S^n$ to the unique nearest 
point of $\Phi(\theta)+T_{\Phi(\theta)}\Phi(S^n)$ from the origin ${\bf 0}$.    
}
\end{definition}
By the definition of Wulff shape, 
if the boundary of a Wulff shape $\partial \mathcal{W}_\gamma$ 
is the image of 
a $C^\infty$ embedding $\Phi: S^n\to \mathbb{R}^{n+1}-\{{\bf 0}\}$, 
the graph of the given $C^\infty$ convex integrand 
$\gamma$
may be considered as the 
pedal relative to the pedal point ${\bf 0}$ for $\Phi$.    
In this case, since graph$(\gamma)$ does not contain the origin, from  
the information of 
$ped_{\Phi, {\bf 0}}: S^n\to \mathbb{R}^{n+1}$, if the boundary of 
a Wulff shape $\partial \mathcal{W}_\gamma$ is the image of 
a $C^\infty$ embedding $\Phi$, the family of affine tangent 
hyperplanes to $\Phi(S^n)$ can be uniquely restored.    
In other words, $ped_{\Phi, {\bf 0}}: S^n\to \mathbb{R}^{n+1}$ 
is one method to store the family of affine tangent 
hyperplanes to $\Phi(S^n)$ if the boundary of 
a Wulff shape $\partial \mathcal{W}_\gamma$ is the image of 
a $C^\infty$ embedding $\Phi$.         
In this sense,  $ped_{\Phi, {\bf 0}}: S^n\to \mathbb{R}^{n+1}$ itself 
may be considered as a sort of Legendre transform for 
the hypersurface $\Phi(S^n)$.      
Since $ped_{\Phi, {\bf 0}}(\theta)=(\theta, \gamma(\theta))$, 
 it follows that if $\Phi(S^n)$ is the graph of 
a $C^\infty$ function $\widehat{\delta}: S^n\to \mathbb{R}_+$, then   
$\gamma$ may be regarded 
as the very Legendre transform of $\widehat{\delta}$.   
Moreover, it has been known that 
$\mathcal{W}_\gamma$ is strictly convex if and only if 
the convex integrand $\gamma$ is of class $C^1$ (\cite{hannishimura2}).   
Thus, in our situation, if both $\gamma$ and $\delta$ are of class 
$C^\infty$, then both of $\mathcal{W}_\gamma$ and $\mathcal{W}_\delta$ 
are strictly convex.    
Therefore, we can expect that 
the Legendre transform works well in our situation.    
\par 
Since Wulff shapes and pedals are defined 
by using perpendicular properties, 
$S^{n+1}$ is more suitable than $\mathbb{R}^{n+1}$ as the space where 
perpendicular properties are considered.    
In the next subsection, we investigate spherical pedals.      

\subsection{Spherical duals and spherical pedals}\label{subsection 2.3}
Let ${\Phi}: S^n\to \mathbb{R}^{n+1}-\{{\bf 0}\}$ be 
a $C^\infty$ embedding.     
We first construct a  $C^\infty$ embedding {\color{black}$S^n\to S^{n+1}$ 
from} the given $\Phi$.   
Let 
$Id: \mathbb{R}^{n+1}\to 
\mathbb{R}^{n+1}\times \{1\}\subset \mathbb{R}^{n+2}$ 
be the mapping defined by 
$Id(x)=(x,1)$.   
Let $N$ be the north pole of $S^{n+1}$ where $S^{n+1}$ is the unit sphere 
in $\mathbb{R}^{n+2}$, namely, 
$N=(0, \ldots, 0, 1)\in S^{n+1}\subset \mathbb{R}^{n+2}$.  
Let $S^{n+1}_{N,+}$ be the northern hemisphere of $S^{n+1}$, namely, 
$S^{n+1}_{N,+}=\{P\in S^{n+1}\; |\; N\cdot P> 0\}$ 
where $N\cdot P$ stands for the 
standard scalar product 
of $(n+2)$-dimensional two vectors $N, P\in \mathbb{R}^{n+2}$.   
Define the mapping 
$\alpha_N: S^{n+1}_{N,+}\to 
\mathbb{R}^{n+1}\times \{1\}\subset \mathbb{R}^{n+2}$, 
called the {\it central projection}, as follows: 
\[
\alpha_N(P_1, \ldots, P_{n+1}, P_{n+2})
=\left(\frac{P_1}{P_{n+2}}, \ldots, \frac{P_{n+1}}{P_{n+2}}, 1\right), 
\] 
where $P=(P_1, \ldots, P_{n+1}, P_{n+2})\in S^{n+1}_{N, +}$.   
Then, the mapping 
$\widetilde{\Phi}: S^n\to S^{n+1}_{N,+}\subset S^{n+1}$ is defined 
as follows.   
\[
\widetilde{\Phi}=\alpha_N^{-1}\circ Id\circ \Phi.   
\] 
\begin{definition}\label{definition 2.3}
{\rm 
For the constructed $C^\infty$ embedding 
$\widetilde{\Phi}: S^n\to S^{n+1}_{N,+}$,  
the {\it spherical pedal} relative to the {\it pedal point} $N$ 
for $\widetilde{\Phi}$, 
denoted by $s$-$ped_{\widetilde{\Phi}, N}: S^n\to S^{n+1}$,  
is the mapping which maps $\theta\in S^n$ to the unique nearest 
point of $GH_{\widetilde{\Phi}(\theta)}\widetilde{\Phi}(S^n)$ 
from the north pole $N$.    
Here, $GH_{\widetilde{\Phi}(\theta)}\widetilde{\Phi}(S^n)$ 
stands for the great 
hypersphere which is tangent to $\widetilde{\Phi}(S^n)$ 
at $\widetilde{\Phi}(\theta)$. 
}
\end{definition}
\par 
Next, we decompose $s$-$ped_{\widetilde{\Phi}, N}$ 
into two simple mappings.    
In order to do so, we firstly define the spherical dual 
$D\widetilde{\Phi}: S^n\to S^{n+1}$ of 
$\widetilde{\Phi}$.    
\begin{definition}\label{spherical dual}
{\rm 
For any $\theta\in S^n$, $D\widetilde{\Phi}(\theta)$ is the point in 
$S^{n+1}_{N,+}$ such that $D\widetilde{\Phi}(\theta)$ is perpendicular 
to any $P\in GH_{\widetilde{\Phi}(\theta)}\widetilde{\Phi}(S^n)$.  
The mapping $D\widetilde{\Phi}: S^n\to S^{n+1}$ is called 
the {\it spherical dual of $\widetilde{\Phi}$}.   
}
\end{definition}
\noindent 
The notion of spherical dual for a spherical curve 
was firstly introduced by Arnol'd 
in \cite{arnold spherical curve}.   
Definition \ref{spherical dual} is a natural generalization 
of his notion to spherical hypersurfaces.   
\smallskip 
\par 
Let $\Psi_N:S^{n+1}-\{\pm N\}\to S^{n+1}$ be the mapping defined by 
\[
\Psi_N(P)=\frac{1}{\sqrt{1-(N\cdot P)^2}}(N-(N\cdot P)P).    
\]
\begin{proposition}[\cite{nishimura}]\label{proposition 2}
$s$-$ped_{\widetilde{\Phi}, N}=\Psi_N\circ D\widetilde{\Phi}$.  
\end{proposition}
In \cite{nishimura}, Proposition \ref{proposition 2} has been proved for 
spherical pedal curves.    However, the proof given in \cite{nishimura} 
works well also for spherical pedal 
hypersurfaces.     
The mapping $\Psi_N$ has the following characteristic properties.      
\begin{enumerate}
\item For any $P\in S^{n+1}-\{\pm N\}$, 
the equality $P\cdot \Psi_N(P)=0$ holds,    
\item for any $P\in S^{n+1}-\{\pm N\}$, 
the property $\Psi_N(P)\in \mathbb{R}N+\mathbb{R}P$ holds,     
\item for any $P\in S^{n+1}-\{\pm N\}$, 
the property $N\cdot \Psi_N(P)>0$ holds,      
\item the restriction 
$\Psi_N|_{S^{n+1}_{N,+}-\{N\}}: S^{n+1}_{N,+}-\{N\}\to 
S^{n+1}_{N,+}-\{N\}$ is a $C^\infty$ diffeomorphism.   
\end{enumerate}
By these properties, the mapping $\Psi_N$ is called 
the {\it spherical blow-up relative to} $N$.   
The mapping $\Psi_N$ is quite useful for studying 
many topics related to perpendicularity.    
For instance, it was used for studying singularities of 
spherical pedal curves in \cite{nishimura, nishimura2}, 
for studying spherical pedal unfoldings in \cite{nishimura3}, 
for studying hedgehogs in \cite{nishimurasakemi}, 
for studying (spherical) Wulff shapes 
in \cite{nishimurasakemi2, hannishimura1, hannishimura2}, 
and for studying 
the aperture of plane curves in \cite{kagatsumenishimura}. 
There is also a hyperbolic version of $\Psi_N$ (\cite{izumiyatari}).  
By the above properties, the following clearly holds: 
\begin{lemma}[\cite{kagatsumenishimura}]
The mapping $Id^{-1}\circ \alpha_N\circ \Psi_N\circ \alpha_N^{-1}\circ Id$ 
is exactly the inversion  
\mbox{\rm inv :}\;$\mathbb{R}^{n+1}- \{{\bf 0}\}
\to \mathbb{R}^{n+1}- \{{\bf 0}\}$.   
\label{inversion}
\end{lemma}
\subsection{Spherical polar sets}\label{subsection 2.4}
In the Euclidean space $\mathbb{R}^{n+1}$, 
the notion of polar set seems to be  
relatively common (for instance, see \cite{matousek}).     
On the other hand, the notion of spherical polar set 
seems to be less common.   
Since the notion of spherical polar set plays an important role in this paper, 
in this subsection, properties of spherical polar sets in $S^{n+1}$ 
are quickly reviewed.   
\par 
For any point $P$ of $S^{n+1}$, we let $H(P)$ be the following set:
\[
H(P)=\{Q\in S^{n+1}\; |\; P\cdot Q\ge 0\}.
\]
\begin{definition}
{\rm 
Let $X$ be a subset of $S^{n+1}$.   Then, the set 
\[
\bigcap_{P\in X}H(P)
\] 
is called the {\it spherical polar set of} $X$ and is denoted by $X^\circ$.   
}
\end{definition}  
By definition, it is clear that $X^\circ$ 
is closed 
for any $X\subset S^{n+1}$.      
\begin{lemma}[\cite{nishimurasakemi2}]\label{lemma 2.1}
Let $X, Y$ be subsets of $S^{n+1}$.    
Suppose that the inclusion $X\subset Y$ holds.   
Then, the inclusion $Y^\circ \subset X^\circ$ holds.   
\end{lemma}
\begin{lemma}[\cite{nishimurasakemi2}]\label{lemma 2.2}
For any subset $X$ of $S^{n+1}$, 
the inclusion $X\subset X^{\circ\circ}$ holds.  
\end{lemma}
\begin{definition}\label{definition 2.1}
{\rm 
A subset $X\subset S^{n+1}$ is said to be 
{\it hemispherical} if there exists a point $P\in S^{n+1}$ such that 
$H(P)\cap X=\emptyset$.   
}
\end{definition}
Let $X$ be a hemispherical subset of $S^{n+1}$.   
Then, for any $P, Q\in X$, $PQ$ stands for the following arc:   
\[
PQ=\left\{\left.\frac{(1-t)P+tQ}{||(1-t)P+tQ||}\in S^{n+1}\; \right|\; 
0\le t\le 1\right\}.  
\] 
Notice that $||(1-t)P+tQ||\ne 0$ for any $P, Q\in X$ 
and any $t\in [0,1]$ if $X\subset S^{n+1}$ is 
hemispherical.       
\begin{definition}\label{definition 2.2}
{\rm 
\begin{enumerate}
\item[(1)]\quad  
A hemispherical subset $X\subset S^{n+1}$ is said to be 
{\it spherical convex} if $PQ\subset X$ for any $P, Q\in X$.     
\item[(2)]\quad 
A hemispherical subset $X\subset S^{n+1}$ is said to be 
{\it strictly spherical convex} if $PQ-\{P,Q\}$ is a subset 
of the set consisting of interior points of $X$ for any $P, Q\in X$.     
\end{enumerate}
}
\end{definition}
Notice that $X^\circ$ is spherical convex 
if $X$ is hemispherical and has an interior point.   
However, in general, 
$X^\circ$ is not necessarily spherical convex 
even if $X$ is hemispherical (for instance if 
$X=\{P\}$ then $X^\circ=H(P)$ is not spherical convex).          
\begin{lemma}[\cite{nishimurasakemi2}]\label{lemma 2.3}
Let $X_\lambda\subset S^{n+1}$ be a spherical convex subset 
for any $\lambda\in \Lambda$ .   
Then, the intersection 
$\cap_{\lambda\in \Lambda}X_\lambda$ is spherical convex.    
\end{lemma}
\begin{definition}\label{definition 2.3}
{\rm Let $X$ be a hemispherical subset of $S^{n+1}$.     
Then, 
the following set is called the {\it spherical convex hull} of $X$ 
and is denoted by 
$\mbox{\rm s-conv}(X)$.   
$$
\mbox{\rm s-conv}(X)= 
\left\{\left.
\frac{\sum_{i=1}^k t_iP_i}{||\sum_{i=1}^kt_iP_i||}\;\right|\; 
P_i\in X,\; \sum_{i=1}^kt_i=1,\; t_i\ge 0,\; k\in \mathbb{N}
\right\}.
$$ 
}
\end{definition}
It is clear that $\mbox{\rm s-conv}(X)=X$ if $X$ is spherical convex.     
More generally, we have the following:   
\begin{lemma}[\cite{nishimurasakemi2}]\label{lemma 2.4}
For any hemispherical subset $X$, 
the spherical convex hull of $X$ is 
the smallest spherical convex set containing $X$. 
\end{lemma}
\begin{definition}\label{definition 2.4}
{\rm 
Let $\{P_1, \ldots, P_k\}$ be a hemispherical finite subset of $S^{n+1}$.   
Suppose that $\mbox{\rm s-conv}(\{P_1, \ldots, P_k\})$ 
has an interior point.    
Then, $\mbox{\rm s-conv}(\{P_1, \ldots, P_k\})$ is called 
the {\it spherical polytope} 
generated by $P_1, \ldots, P_k$.      
}
\end{definition}
\begin{proposition}[\cite{gaohugs, nishimurasakemi2}]\label{proposition 2.1}
For any closed hemispherical subset $X\subset S^{n+1}$, 
the equality $\mbox{\rm s-conv}(X)
=(\mbox{\rm s-conv}(X))^{\circ\circ}$ holds.
\end{proposition}
Notice that for any closed hemispherical subset 
$X\subset S^{n+1}$, $\mbox{\rm s-conv}(X)$, too, is 
closed and hemispherical.   
Notice also that by Lemma \ref{lemma 2.2}, 
for any subset $X\subset S^{n+1}$, the inclusion 
$X\subset X^{\circ\circ}$ always holds.   However, 
the inverse inclusion 
$X\supset X^{\circ\circ}$ does not hold in general 
even if $X$ is closed and hemispherical.  
\begin{lemma}[\cite{nishimurasakemi2}]
For any hemispherical finite subset $X=\{P_1, \ldots, P_k\}\subset S^{n+1}$, 
the following holds:
$$
\left\{\left.
\frac{\sum_{i=1}^k t_iP_i}{||\sum_{i=1}^kt_iP_i||}\;\right|\; 
P_i\in X,\; \sum_{i=1}^kt_i=1,\; t_i\ge 0
\right\}^\circ 
= 
H(P_1)\cap \cdots \cap H(P_k).
$$
\label{lemma 2.5}
\end{lemma}
\noindent 
Lemma \ref{lemma 2.5} is called {\it Maehara's lemma}.  
\subsection{Caustics and symmetry sets}\label{subsection 2.5}
Let $\Phi:S^n\to\R^{n+1}$ be a $C^\infty$ embedding. 
Consider the family of functions $F:\R^{n+1}\times S^n\to \R$ defined by
\[
F(v,\theta)=\dfrac{1}{2}||\Phi(\theta)-v||^2.
\]
Notice that $F$ may be regarded as {\color{black}the} map 
from $\R^{n+1}$ to $C^\infty(S^n,\R)$ 
which maps each $v\in\R^{n+1}$ to the function 
$f_v(\theta)=F(v,\theta)\in C^\infty(S^n,\R)$.    
The set consisting of vectors $v$ for which $f_v(\theta)$ has 
a degenerate critical point form the \emph{Caustic} of $\Phi$, 
denoted by $Caust(\Phi)$ (or $Caust(\Phi(S^n))$).     
For details on caustics, see for instance 
\cite{arnold, arnolddynamical8, arnold-guseinzade-varchenko, 
izumiya1, izumiya-takahashi}.  
The set consisting of vectors $v$ for which $f_v(\theta)$ has 
a multiple critical value form the \emph{Symmetry Set} of $\Phi$, 
denoted by $Sym(\Phi)$ (or $Sym(\Phi(S^n))$). 
For details on symmetry sets, see for instance 
\cite{bruce-giblin1, bruce-giblin2, bruce-giblin-gibson}.     
These two sets provide the set consisting of vectors $v$ 
at which the function  $f_v(\theta)$ is not stable.    
%
\subsection{Spherical Wulff shapes, spherical caustics 
and spherical symmetry sets}\label{subsection 2.6}
Let $\mathcal{W}_\gamma$ be a Wulff shape.    
Then, the image of $\mathcal{W}_\gamma$ 
by $\alpha_N^{-1}\circ Id: \mathbb{R}^{n+1}\to S^{n+1}_{N,+}$ is 
called the {\it spherical Wulff shape} 
associated with $\mathcal{W}_\gamma$   
and is denoted by $\widetilde{\mathcal{W}}_\gamma$.   
By using the spherical blow-up and the spherical polar set operation, 
any spherical Wulff shape $\widetilde{\mathcal{W}}_\gamma$ can be 
characterized as follows:   
\[
\widetilde{\mathcal{W}}_\gamma 
= 
\left(\Psi_N\circ \alpha_N^{-1}\circ Id
\left(\mbox{\rm graph}(\gamma)\right)\right)^\circ.   
\] 
For a spherical Wulff shape $\widetilde{\mathcal{W}}_\gamma$, 
the spherical polar set 
$\left(\widetilde{\mathcal{W}}_\gamma\right)^\circ$ is called the 
\emph{spherical dual} of $\widetilde{\mathcal{W}}_\gamma$, 
and is denoted by 
$\mathcal{D}\widetilde{\mathcal{W}}_\gamma$.   
\[
\mathcal{D}\widetilde{\mathcal{W}}_\gamma
= 
\left(\widetilde{\mathcal{W}}_\gamma\right)^\circ.   
\]
{\color{black} 
\begin{lemma}\label{lemma 2.7}
Let $\gamma:S^n\to \mathbb{R}_+$ be a convex integrand and 
let $\delta: S^n\to \mathbb{R}_+$ be the dual of $\gamma$.    
Then, the following hold:   
\begin{enumerate}
\item[(1)]\quad $\mathcal{D}\mathcal{W}_\gamma=
Id^{-1}\circ\alpha_N(\mathcal{D}\widetilde{\mathcal{W}}_\gamma)$.    
\item[(2)]\quad $\mathcal{DDW}_\gamma=\mathcal{W}_\gamma$.   
\item[(3)]\quad The dual of $\delta$ is $\gamma$.   
\end{enumerate}
\end{lemma}
\proof 
We first show the assertion (1).   
By definition, it is clear that the following sublemma holds.  
\begin{sublemma}\label{sublemma 2.1}
For any $P\in S^{n+1}_{N,+}$, the following equality holds where 
$\theta\in S^n$ and $r\in \mathbb{R}_+$ is defined by 
$(\theta, r)=Id^{-1}\circ \alpha_N\circ\Psi_N(P)$.   
\[
Id^{-1}\circ \alpha_N\left(H(P)\cap S^{n+1}_{N,+}\right)
= 
\{x\in \mathbb{R}^{n+1}\; |\; x\cdot \theta\le r\}.   
\]
\end{sublemma}
We have the following:      
\begin{eqnarray*}
{ } & { } & 
Id^{-1}\circ \alpha_N(\mathcal{D}\widetilde{W}_\gamma) \\ 
{ } & = & 
Id^{-1}\circ \alpha_N
\left(\bigcap_{P\in\widetilde{\mathcal{W}}_\gamma}H(P)\right) \\ 
{ } & = & 
Id^{-1}\circ \alpha_N
\left(\bigcap_{P\in\partial \widetilde{\mathcal{W}}_\gamma}H(P)\right) 
\quad (\mbox{by Lemma \ref{lemma 2.5}}) \\
 { } & = & 
\bigcap_{\theta\in S^n}
Id^{-1}\circ \alpha_N
\left(
H\left(\alpha_N^{-1}\circ Id\left(\theta, \hat{\delta}(\theta)\right)\right) 
\cap S^{n+1}_{N,+}
\right)
\quad (\mbox{since }Id^{-1}\circ \alpha_N\mbox{ is bijective.}) \\ 
 { } & = & 
\bigcap_{\theta\in S^n}
\left\{
x\in \mathbb{R}^{n+1}\; |\; x\cdot \theta\le \delta(\theta)
\right\}
\quad (\mbox{by Lemma \ref{inversion} 
and Sublemma \ref{sublemma 2.1}}) \\ 
 { } & = &  
 \mathcal{W}_\delta=\mathcal{DW}_\gamma.   
\end{eqnarray*}
Hence, the assertion (1) follows.   
\par 
Nextly, we show the assertion (2).    
By Proposition \ref{proposition 2.1}, the following holds:   
\[
\mathcal{DD}\widetilde{\mathcal{W}}_\gamma = 
\left(\widetilde{W}_\gamma\right)^{\circ\circ}
= \widetilde{W}_\gamma.   
\]
Hence, by the assertion (1), we have the following, which proves 
the assertion (2): 
\[
\mathcal{DDW}_\gamma 
= 
Id^{-1}\circ 
 \alpha_N\left(\mathcal{DD}\widetilde{\mathcal{W}}_\gamma\right) 
= 
Id^{-1}\circ \alpha_N\left(\widetilde{\mathcal{W}}_\gamma\right) 
=
\mathcal{W}_\gamma.   
\]   
\par 
Finally we show the assertion (3).      
Let $\xi: S^n\to \mathbb{R}_+$ be the dual of $\delta$.    
Then, by the assertion (2), we have the following:   
\[
\mathcal{W}_\xi=
\mathcal{DW}_\delta 
= 
\mathcal{DDW}_\gamma 
=
\mathcal{W}_\gamma.  
\]
The following has been known:   
\begin{proposition}[\cite{nishimurasakemi2, taylor}]
\label{nishimurasakemi2_taylor}
For any two continuous functions 
$\gamma_1, \gamma_2: S^n\to \mathbb{R}_+$, 
the equality 
$\Gamma_{\gamma_1}=\Gamma_{\gamma_2}$ 
holds if and only if 
the equality 
$\mathcal{W}_{\gamma_1}=\mathcal{W}_{\gamma_2}$ holds.
\end{proposition}
\noindent 
By Proposition \ref{nishimurasakemi2_taylor}, we have the following:   
\[
\Gamma_{\xi}=\Gamma_{\gamma}.
\]
Since both $\xi, \gamma$ are convex integrands, 
it follows that $\xi=\gamma$.    
\hfill $\square$  
}
\par 
\medskip 
Suppose that the boundary of 
$\widetilde{\mathcal{W}}_\gamma$ is the image 
of a $C^\infty$ embedding 
$\widetilde{\Phi}: S^n\to S^{n+1}_{N,+}{\color{black}-\{N\}}$.    
Then, for the embedding $\widetilde{\Phi}$, 
the \emph{Spherical Caustic} 
and {\color{black}the} \emph{Spherical Symmetry Set} 
can be defined as follows.   
Let $d:S_{N,+}^{n+1}\times S_{N,+}^{n+1}\to \R$ be the distance squared 
function, 
i.e., $d\left({P}_1,{P}_2\right)$ is 
the square of the length of the arc $P_1P_2$.   
Consider the family of functions 
$\widetilde{F}:S^{n+1}_{N,+}\times S^n\to \R$ defined by
\[
\widetilde{F}(v,\theta)=\dfrac{1}{2}d\left(\widetilde{\Phi}(\theta), v\right).
\]
Then, $\widetilde{F}$ may be regarded as {\color{black}the} map from 
$S^{n+1}_{\color{black}N,+}$ to $C^\infty(S^n,\R)$ 
which maps each $v\in S^{n+1}_{\color{black}N,+}$ to the function 
$\widetilde{f}_v(\theta)=\widetilde{F}(v,\theta)\in C^\infty(S^n,\R)$.    
The set consisting of vectors $v$ for which $\widetilde{f}_v(\theta)$ has 
a degenerate critical point form the \emph{Spherical Caustic} of 
$\widetilde{\Phi}$, 
denoted by {\it Sph-Caust}$(\widetilde{\Phi})$ 
(or {\it Sph-Caust}$(\widetilde{\Phi}(S^n))$).    
The set consisting of vectors $v$ for which $\widetilde{f}_v(\theta)$ has 
a multiple critical value form the \emph{Spherical Symmetry Set} of 
$\widetilde{\Phi}$, 
denoted by {\it Sph-Sym}$(\widetilde{\Phi})$ 
(or {\it Sph-Sym}$(\widetilde{\Phi}(S^n))$). 
%
\par 
{\color{black}Moreover, for} any $t\in \mathbb{R}$ 
{\color{black}such that $|t|<\pi$}, 
we define a $C^\infty$ mapping, 
denoted by $\widetilde{\Phi}_t:S^n\to S^{n+1}$, as follows.    
For any $\theta\in S^n$, let $GC_{\widetilde{\Phi}(\theta)}$ is 
the great circle passing through $\widetilde{\Phi}(\theta)$ which is 
perpendicular to $\widetilde{\Phi}(S^n)$ 
at $\widetilde{\Phi}(\theta)$.       
For any $t\in \mathbb{R}$ {\color{black}$(0<|t|<\pi)$}, 
inside $GC_{\widetilde{\Phi}(\theta)}$, there exist 
exactly two {\color{black}distinct} points 
$P_1(\theta), P_2(\theta)$ such that 
$d(P_1(\theta), \widetilde{\Phi}(\theta))=
d(P_2(\theta), \widetilde{\Phi}(\theta))=t^2$.    
Notice that one of $P_1(\theta), P_2(\theta)$ is 
inside the connected component 
of $S^{n+1}-\widetilde{\Phi}(S^n)$ containing $N$.    
Without loss of generality, we may assume that $P_1(\theta)$ is 
inside the region.   
Then, for any $t$ {\color{black}$(0<|t|<\pi)$}, 
the mapping $\widetilde{\Phi}_t:S^n\to S^{n+1}$ is defined by 
$\widetilde{\Phi}_t(\theta)=P_1(\theta)$ 
(resp., $\widetilde{\Phi}_t(\theta)=P_2(\theta)$) if $t$ is positive 
(resp., $t$ is negative).       
For $t=0$, the point $\widetilde{\Phi}_0(\theta)$ 
is defined as the point $\widetilde{\Phi}(\theta)$.   
The mapping $\widetilde{\Phi}_t:S^n\to S^{n+1}$ is called 
the \emph{spherical wave front} of $\widetilde{\Phi}$.    
It is clear that the spherical wave front of $\widetilde{\Phi}$ is 
a $C^\infty$ mapping for any $t$ {\color{black}such that $|t|<\pi$}.    
Notice that 
$\widetilde{\Phi}_{\pi/2}$ is exactly ${D}\widetilde{\Phi}$.     
It follows that ${DD}\widetilde{\Phi}=\widetilde{\Phi}$.   
It is easily seen that the following proposition holds (see also 
FIGURE 1 where the dot dash curve is the spherical 
caustic, dotted curves are the spherical wave fronts.  
Two images by orthogonal projections also are depicted 
for the sake of clearness.).
\begin{figure}[ht]
 \begin{center}
 \includegraphics[scale=0.7]{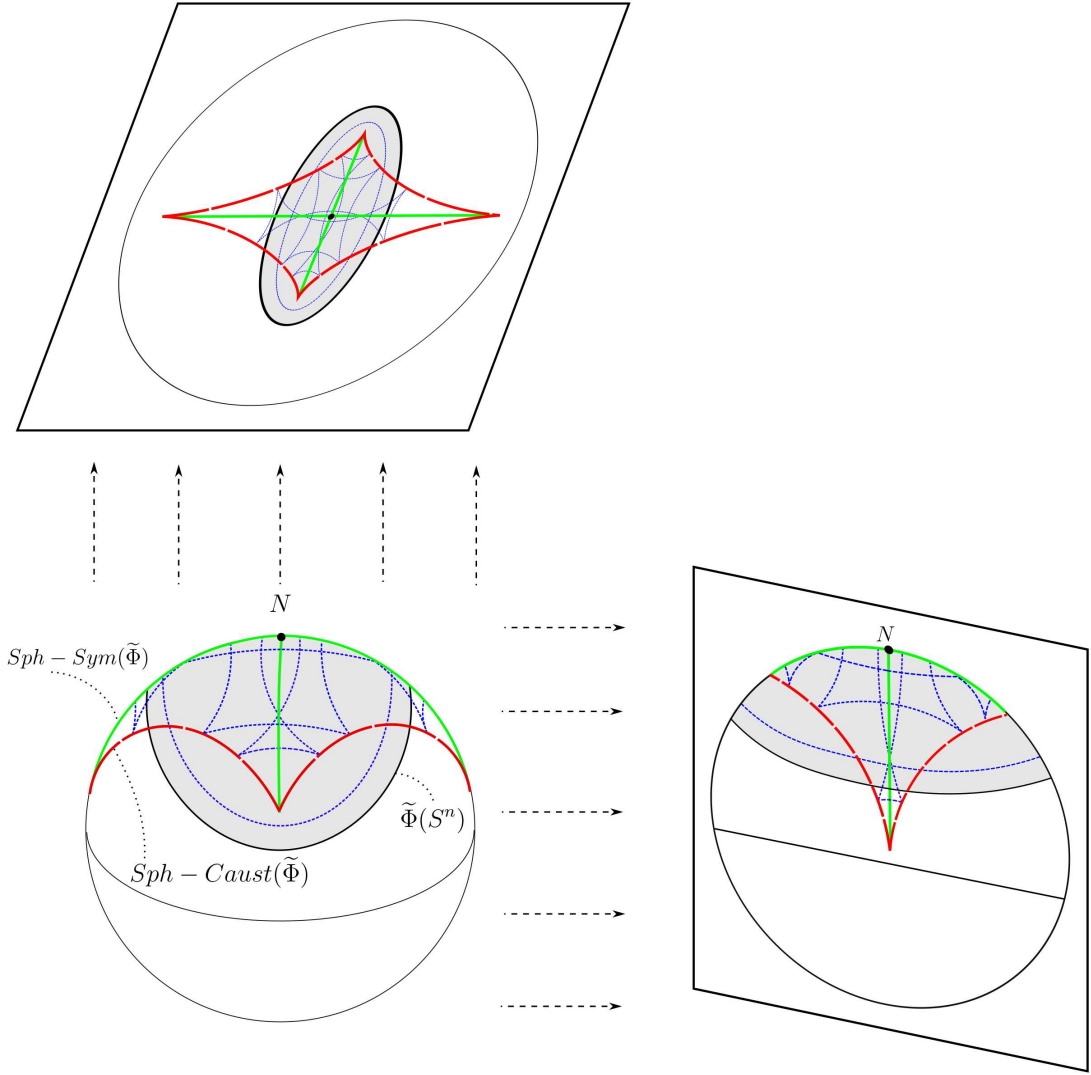}
 \end{center}
\caption{Spherical caustic, spherical symmetry set 
and spherical wave fronts.}
\label{figure 1}
\end{figure}
\begin{proposition}\label{proposition wave front}
\begin{enumerate}
\item[(1)]\quad 
\[
Sph\mbox{-}Caust\left(\widetilde{\Phi}\right) = 
\bigcup_{\color{black}|t|<\pi}\left.\left\{\widetilde{\Phi}_t(\theta)\; \right|\; 
\theta \mbox{\rm \; is a singular point of }\widetilde{\Phi}_t \right\}. 
\]
\item[(2)]\quad 
\[
Sph\mbox{-}Sym\left(\widetilde{\Phi}\right) = 
\bigcup_{\color{black}|t|<\pi}
\left.\left\{\widetilde{\Phi}_t(\theta_1)=\widetilde{\Phi}_t(\theta_2)
\; \right|\; 
\theta_1, \theta_2\in S^n\; (\theta_1\ne \theta_2)  \right\}. 
\]
\end{enumerate}
\end{proposition}
By Proposition \ref{proposition wave front}, 
if we consider inside the sphere $S^{n+1}$, 
it is expected that everything is clearly understood.   
Moreover, by \cite{hannishimura4}, 
the angle $\pi/2$ is {\color{black}closely} related to 
the self-dual Wulff shapes.    Thus, we may consider that 
$\pi/2$ is a  significant number for studying Wulff shapes, 
although we have no such significant real numbers 
if we restrict ourselves to consider Wulff shapes 
only in $\mathbb{R}^{n+1}$.    
\begin{definition}\label{legendrian}
{\rm 
A $C^\infty$ map-germ 
$f: (\mathbb{R}^n, {\bf 0})\to (\mathbb{R}^{n+1}, {\bf 0})$ 
is said to be {\it Legendrian} 
if there exists a germ of $C^\infty$ vector field 
$\nu_f$ along $f$ satisfying the following two:
\begin{enumerate}
\item[(1)]\quad 
\[
\frac{\partial f}{\partial x_1}(x)\cdot \nu_f(x)=\cdots = 
\frac{\partial f}{\partial x_n}(x)\cdot \nu_f(x)=0.   
\]
\item[(2)]\quad 
The map-germ 
$L_f: (\mathbb{R}^n, {\bf 0})\to T_1\mathbb{R}^{n+1}$ defined as follows 
is non-singular, where $T_1\mathbb{R}^{n+1}$ is 
the unit tangent bundle of $\mathbb{R}^{n+1}$.   
\[
L_f(x)=\left(f(x), \nu_f(x)\right).
\]
\end{enumerate}   
 }
\end{definition}  
\noindent 
It is well-known that for any $t$ {\color{black}such that $|t|<\pi$} and any 
$\theta\in S^n$, the germ of spherical wave front 
$\widetilde{\Phi}_t: (S^n, \theta)\to {\color{black}S^{n+1}}-\{N\}$ 
is Legendrian.   
For details on Legendrian map-germs, see for instance 
\cite{arnold, arnolddynamical8, arnold-guseinzade-varchenko, 
izumiya1, izumiya-takahashi}.   
\subsection{Andrews formulas}\label{subsection 2.7}
Let $\gamma:S^n\to\R_+$ be 
a $C^\infty$ {\color{black}strictly} convex integrand 
and let $\delta:S^n\to\R_+$ 
be the dual convex integrand of $\gamma$.   
Notice that $\partial \mathcal{W}_\gamma$ 
(resp., $\partial \mathcal{DW}_\gamma$) 
is the image of 
the embedding $\Phi_\delta$ (resp., $\Phi_\gamma$)  
defined by 
$\Phi_\delta(\theta)=\left(\theta, \widehat{\delta}(\theta)\right)$ (resp.,  
$\Phi_\gamma(\theta)=\left(\theta, \widehat{\gamma}(\theta)\right)$).  
%
Define the mapping $h_\delta:S^n\to S^n$ 
(resp., $h_\gamma:S^n\to S^n$) so that  
$\gamma(\theta)$ (resp., $\delta(\theta)$) 
is the perpendicular distance 
from the origin ${\bf 0}$ 
to the affine tangent hyperplane to $\Phi_\delta(S^n)$ 
(resp., $\Phi_\gamma(S^n)$) 
at $\Phi_\delta(h_\delta(\theta))$ 
(resp., $\Phi_\gamma(h_\gamma(\theta))$).       
Then, it turns out that both of $h_\delta, h_\gamma$ 
are $C^\infty$ diffeomorphisms 
(see Remark \ref{remark 3.1} in Section \ref{section 3})
\footnote{\color{black}Subsection \ref{subsection 2.7} is used 
only for the proof of Theorem \ref{theorem 3}.   For the proofs of 
Theorems \ref{theorem 1} and \ref{theorem 2}, 
Subsection \ref{subsection 2.7} 
is unnecessary.}.     
Moreover, {\color{black}since} both of $\Phi_\delta(S^n), \Phi_\gamma(S^n)$ 
are strictly convex {\color{black}by \cite{hannishimura2}}, 
{\color{black}as shown in \cite{andrews},  
the simultaneous equations for envelopes 
leads to}  
the following equalities for any $\theta\in S^n$ 
(see FIGURE \ref{Andrews's formula} where $T_1$ (resp., $T_2$) denotes  
the affine tangent hyperplane to 
$\Phi_\delta(S^n)$ (resp., $\Phi_\gamma(S^n)$) 
at $\Phi_\delta(h_\delta(\theta))$ 
(resp., $\Phi_\gamma(h_\gamma(\theta))$)). 
\begin{eqnarray*}
\Phi_\delta(h_\delta(\theta)) & 
= & \gamma(\theta)\theta+\nabla \gamma(\theta), \\ 
\Phi_\gamma(h_\gamma(\theta)) & 
= & \delta(\theta)\theta+\nabla \delta(\theta).  
\end{eqnarray*}
Here, $\nabla \gamma(\theta)$ (resp., $\nabla \delta(\theta)$) stands for  
the gradient vector of $\gamma$ (resp., $\delta$) at $\theta\in S^n$  
with respect to the standard metric on $S^n$.     
\begin{figure}[ht]
 \begin{center}
 \includegraphics[scale=0.5]{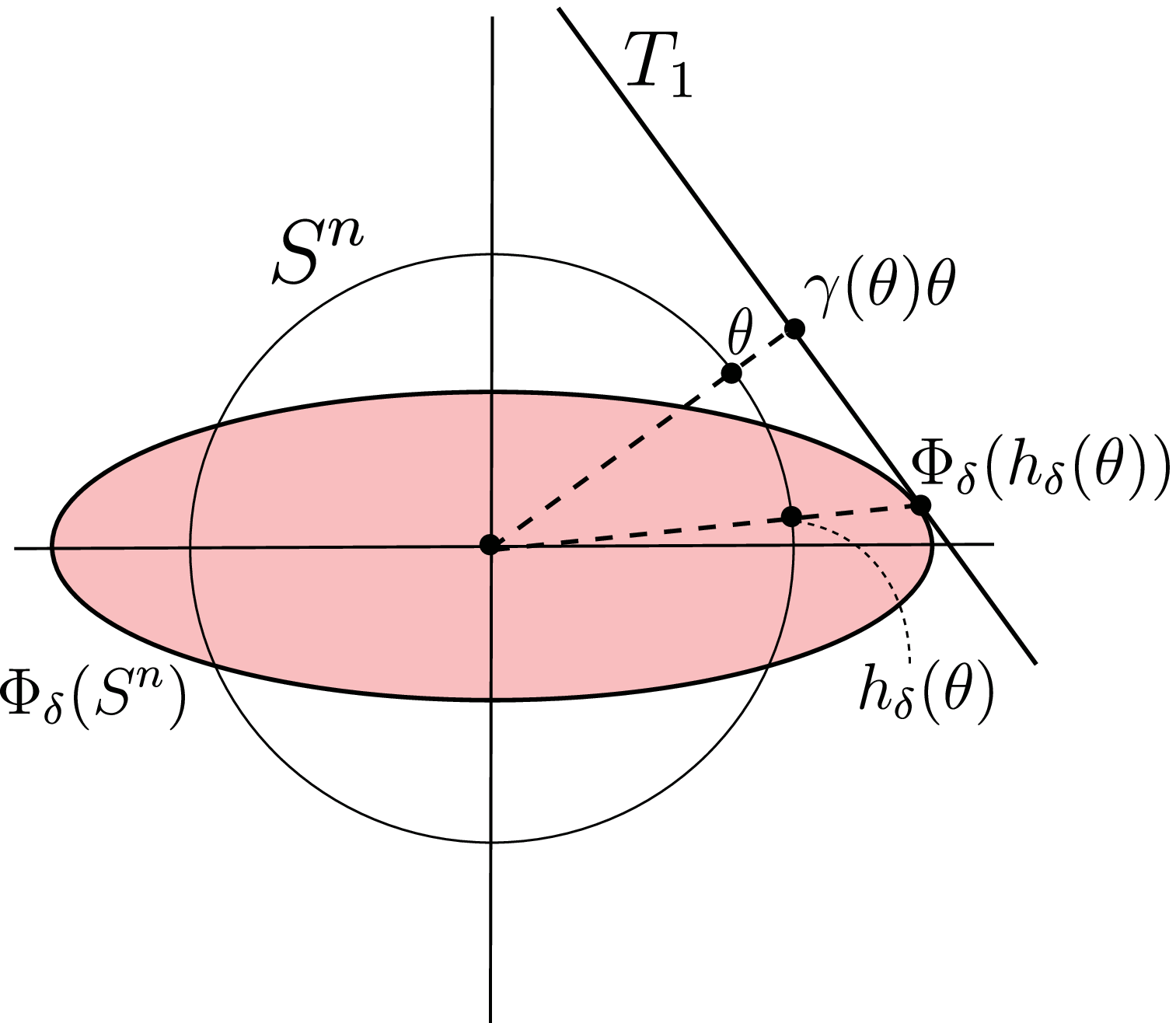}
 \includegraphics[scale=0.5]{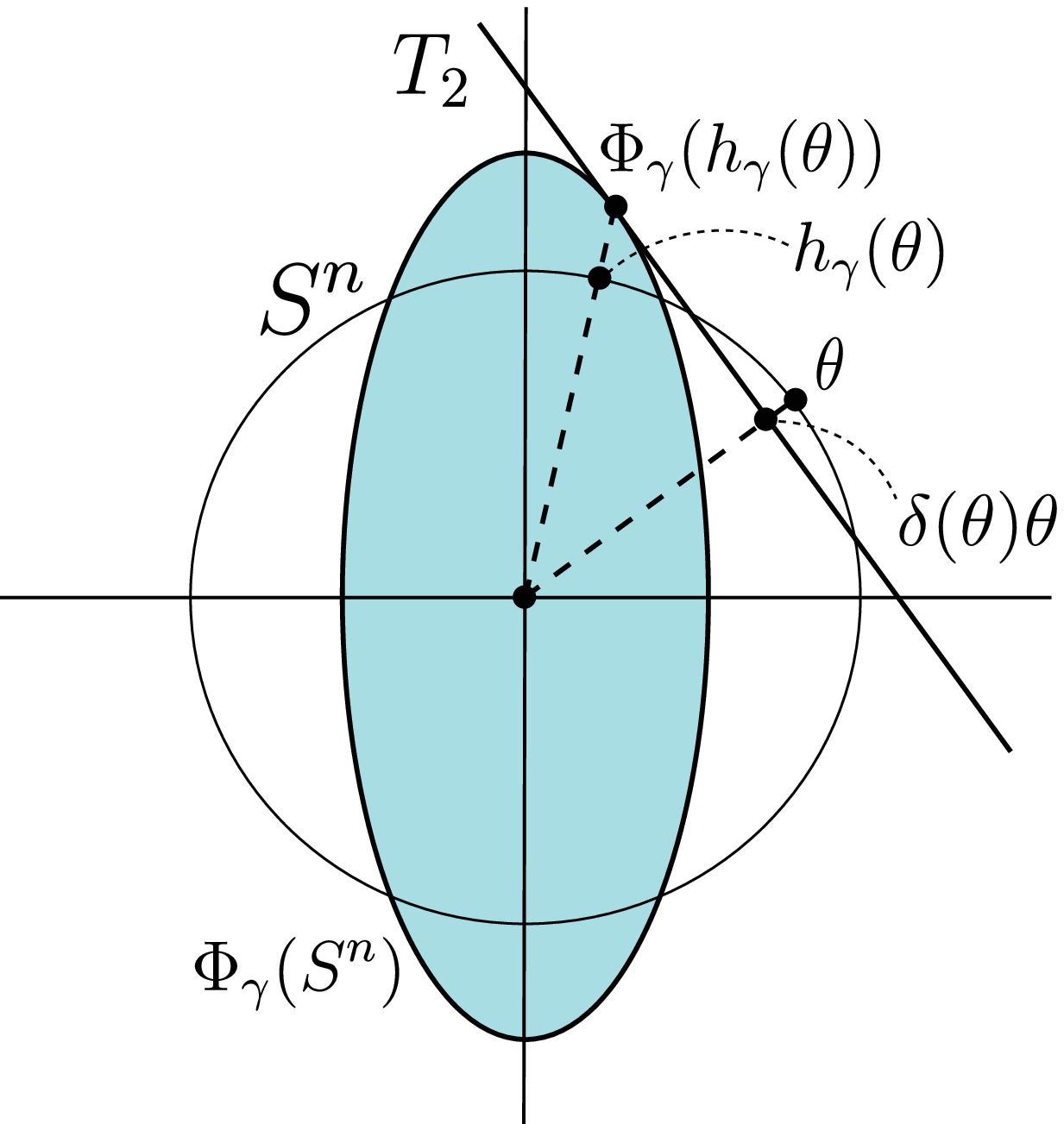}
 \end{center}
\caption{Andrews formulas}\label{Andrews's formula}
\end{figure}
\section{Proof of Theorem \ref{theorem 1}}\label{section 3}
The notations, terminologies and notions introduced in Sections 1 and 
\ref{section 2} 
are used without 
explaining them explicitly again. 
{\color{black}
Firstly, we show that $\delta$ is a $C^{\infty}$ convex integrand under the
 assumption that its dual $\gamma$ is a $C^{\infty}$ 
 strictly convex integrand.
}
\par 
Define the $C^\infty$ embedding 
$\Psi: S^n\to \mathbb{R}^{n+1}-\{0\}$ by 
\[
\Psi(\theta)=(\theta, \gamma(\theta)).   
\]
Since $\delta$ and $\widehat{\delta}$ are $\mathcal{A}$-equivalent, 
it is sufficient to show that the function 
$\widehat{\delta}: S^n\to \mathbb{R}_+$ is of class $C^\infty$ 
so that 
$ped_{\Phi_\delta,{\bf 0}}=\Psi$ where 
$\Phi_\delta$ is the $C^\infty$ embedding 
defined in Subsection \ref{subsection 2.7}
For this purpose, the ambient space is changed to $S^{n+1}$.   
Let $\widetilde{\Psi}_{\widehat{\gamma}}:S^{n}\to S_{N,+}^{n+1}$ 
be the $C^\infty$ embedding defined by 
\[
\widetilde{\Psi}_{\widehat{\gamma}}(\theta)=
\alpha_N^{-1}\circ Id\left(\theta, \widehat{\gamma}(\theta)
\right) \quad (\theta\in S^n).   
\] 
Then, since $\gamma$ is a convex integrand, the image 
$\widetilde{\Psi}_{\widehat{\gamma}}(S^n)$ is the boundary of a 
spherical convex set $\widetilde{\mathcal{W}}_\delta$.    
Set 
\[
\widetilde{\Psi}_{\widehat{\gamma}}(\theta)=
\left(\widetilde{\Psi}_1(\theta), \ldots, 
\widetilde{\Psi}_{n+2}(\theta)\right).    
\]
Define the mapping $\widetilde{\Phi}: S^n\to S^{n+1}$ by 
\[
\widetilde{\Phi}={D}\widetilde{\Psi}_{\widehat{\gamma}}.    
\]
Then, by Proposition \ref{proposition 2}, we have the following: 
\[
s\mbox{-}ped_{\widetilde{\Phi}, {\color{black}N}}
=\Psi_N\circ \widetilde{\Psi}_{\widehat{\gamma}}.     
\]
Set 
\[
\widetilde{\Phi}(\theta)=
\left(\widetilde{\Phi}_1(\theta), \ldots, 
\widetilde{\Phi}_{n+2}(\theta)\right),     
\]
where $\left(\widetilde{\Phi}_1(\theta), \ldots, 
\widetilde{\Phi}_{n+2}(\theta)\right)$ 
is the standard Euclidean expression of 
the point $\widetilde{\Phi}(\theta)$.  
Define the mapping $\widetilde{h}: S^n\to S^n$ by   
\[
\widetilde{h}(\theta)=
\frac{\left(\widetilde{\Phi}_1(\theta), \ldots, 
\widetilde{\Phi}_{n+1}(\theta)\right)}
{\left|\left|\left(\widetilde{\Phi}_1(\theta), \ldots, 
\widetilde{\Phi}_{n+1}(\theta) \right)\right|\right|}.   
\]
Then, by the definition of dual ${D}\widetilde{\Psi}_{\widehat{\gamma}}$, 
it follows that $\widetilde{h}$ is a well-defined $C^\infty$ mapping.

{\color{black}Since $\gamma$ 
is a strictly convex integrand, b}y the two equalities 
${DD}\widetilde{\Psi}_{\widehat{\gamma}}=\widetilde{\Psi}_{\widehat{\gamma}}$ 
and 
\[
\frac{\left(\widetilde{\Psi}_1(\theta), \ldots, 
\widetilde{\Psi}_{n+1}(\theta)\right)}
{\left|\left|\left(\widetilde{\Psi}_1(\theta), \ldots, 
\widetilde{\Psi}_{n+1}(\theta) \right)\right|\right|}
=-\theta, 
\] 
it follows that $\widetilde{h}$ is bijective.   
\par 
Next, we show that $\widetilde{h}$ is a $C^\infty$ diffeomorphism.    
Since $\widetilde{\Phi}$ is the spherical dual of the 
$C^\infty$ embedding 
$\widetilde{\Psi}_{\widehat{\gamma}}: S^n\to S^{n+1}_{N,+}$, 
it is the spherical wave front  
$\left(\widetilde{\Psi}_{\widehat{\gamma}}\right)_{\pi/2}$.   
Thus, $\widetilde{\Phi}$ is Legendrian.     
Hence, for any singular point $\theta_0\in S^n$ of 
$\widetilde{\Phi}$, there exists a germ of $C^\infty$ normal vector 
field $\nu_{\widetilde{\Phi}}$ along $\widetilde{\Phi}$ such that 
the map-germ $L_{\widetilde{\Phi}}: (S^n, \theta_0)\to T_1S^{n+1}$ 
defined as follows 
is non-singular.   
\[
L_{\widetilde{\Phi}}(\theta)= 
\left(
\widetilde{\Phi}(\theta), \nu_{\widetilde{\Phi}}(\theta)
\right).   
\]    
In particular, even at the critical value $\widetilde{\Phi}(\theta_0)$, 
the normal great circle to 
$\widetilde{\Phi}(S^n)$ at $\widetilde{\Phi}(\theta_0)$ 
must be unique.    
By this fact, it is easily seen that the spherical Wulff shape 
$\widetilde{W}_\delta$ is strictly spherical convex.     
Hence, by \cite{hannishimura2}, it follows that 
the image $\widetilde{\Phi}(S^n)$ is the graph of a $C^1$ function.   
This implies that the derivative $d\widetilde{h}_\theta$ is bijective 
for any $\theta\in S^n$.    
Therefore, by the inverse function theorem, 
$\widetilde{h}$ is a $C^\infty$ diffeomorphism.   
\par 
Notice that $\Phi(\theta)$ can be expressed as follows:   
\[
\Phi(\theta)
=Id^{-1}\circ \alpha_N\circ \widetilde{\Phi}\circ \widetilde{h}^{-1}(\theta) 
= 
\left(
\theta, \tan \left(
\cos^{-1}\left(
\widetilde{\Phi}_{n+2}\circ \widetilde{h}^{-1}(\theta)
\right)
\right)
\right).  
\] 
Hence, we have the following:   
\[
\widehat{\delta}(\theta) = 
\tan \left(
\cos^{-1}\left(
\widetilde{\Phi}_{n+2}\circ \widetilde{h}^{-1}(\theta)
\right)
\right).    
\]
Since $\widetilde{\Phi}_{n+2}$ is of class $C^\infty$ and 
all of $\widetilde{h}: S^n\to S^n$, $\cos : (0, \pi/2)\to (0,1)$  
and $\tan : (0,\pi/2)\to \mathbb{R}_+$ are $C^\infty$ diffeomorphisms, 
it follows that 
$\widehat{\delta}$ is of class $C^\infty$.    
\indent
\par
{\color{black}
Next, we show that $\gamma$ is a strictly convex integrand 
under the assumption that 
its dual $\delta$ is a $C^{\infty}$ convex integrand. 
In \cite{hannishimura2}, it is shown that a Wulff shape is 
strictly convex  if and only if its convex integrand 
$S^{n}\to \mathbb{R}_{+}$ is of class $C^{1}$. Since $\delta$ is a 
 $C^{\infty}$ convex integrand, we have that the Wulff shape 
 \[
 \mathcal{W}_{\delta}=\mbox{the convex hull of inv(graph($\gamma$))}
 \] 
 is strictly convex. This implies that the convex integrand 
 $\gamma$ is a $C^{\infty}$ strictly convex integrand.
}
\hfill $\Box$
\begin{remark}\label{remark 3.1}
\begin{enumerate}
\item[(1)]\quad 
Notice that 
$\widetilde{h}: S^n\to S^n$ is exactly the same mapping 
as $h_\delta: S^n\to S^n$ given in Subsection \ref{subsection 2.7}.   
Thus, $h_\delta$ is a $C^\infty$ diffeomorphism.   
Similarly, $h_\gamma$ in Subsection \ref{subsection 2.7} 
also is a $C^\infty$ diffeomorphism.  
{
\color{black}
\item[(2)]\quad 
FIGURE 3 
  explains that $\widetilde{h}: S^n\to S^n$ is not bijective 
if $\gamma$ is not a $C^\infty$ strictly convex integrand but 
a $C^\infty$ convex integrand.    
}

\begin{figure}[ht]
 \begin{center}
 \includegraphics[scale=0.7]{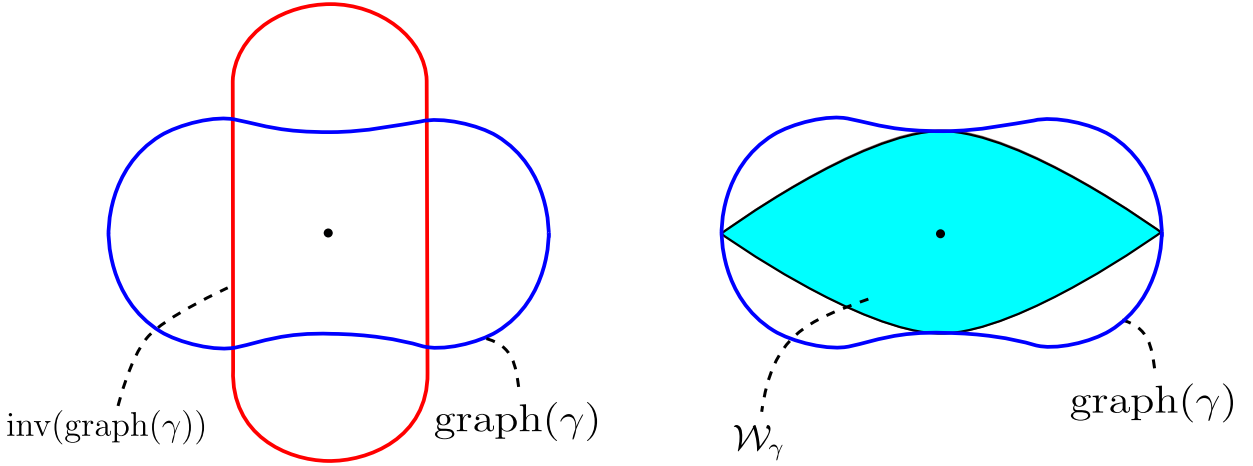}
 \end{center}
{
\color{black}
\caption{LEFT: The image inv(graph($\gamma$) for a 
$C^\infty$ not strictly convex integrand $\gamma$.   
RIGHT: $\mathcal{W}_\gamma$.}
}
\label{not strictly convex integrand}
\end{figure}
\end{enumerate}
\end{remark}
\section{Proof of Theorem \ref{theorem 2}}\label{section 4}
The notations, terminologies and notions introduced in Sections 1,  
\ref{section 2} and \ref{section 3} are used without 
explaining them explicitly again.      
\par 
Define the function 
$\widetilde{\gamma}: S^n\to \mathbb{R}_+$ by $\widetilde{\gamma}
(\theta)=d(\widetilde{\Psi}_{\widehat{\gamma}}(\theta),N)$.    
\begin{lemma}\label{lemma 4.1}
There exists a point $\theta$ of $S^{n}$ which is a degenerate critical 
point of $\widetilde{\gamma}$ 
if and only if 
the north pole 
$N$ is contained 
in the spherical caustic of $\widetilde{\Psi}_{\widehat{\gamma}}$.
\end{lemma}
\begin{proof}
Let $S(\Phi)$ be the set consisting of singular points of the map 
$\Phi$ defined by
\[
\begin{array}{cccccc}
  \Phi: & S^n\times S_{N,+}^{n+1} & \rightarrow & \R\times S_{N,+}^{n+1} \\
  \,\,\, & (\theta,{P}) & \mapsto & 
(d(\widetilde{\Psi}_{\widehat{\gamma}}(\theta),{P}),{P}).   
\end{array}
\]
Then, it is well-known that $S(\Phi)$ is an $(n+1)$-dimensional 
$C^\infty$ submanifold of $S^n\times S_{N,+}^{n+1}$ 
(for instance, see \cite{porteous, romerofusterruas}).       
{\color{black}
Denote the restriction to $S(\Phi)$ of the canonical projection 
$\pi:S^n\times S_{N,+}^{n+1}\to S_{N,+}^{n+1}$ 
by $\Pi$.      
Then, the set $\Pi \left(S(\Pi) \right)$ is
}
the spherical caustic of $\widetilde{\Psi}_{\widehat{\gamma}}$.     
\par 
Notice that when ${P}=N$,  
we have that $\Phi(\theta,N)=(\widetilde{\gamma}(\theta),N)$.    
Thus, we have the following:   
\[
N\in Spherical\mbox{-}
Caust(\widetilde{\Psi}_{\widehat{\gamma}})\Leftrightarrow \exists 
\theta\in S^n \mbox{\rm \; such that } \nabla\widetilde{\gamma}
(\theta)=0\,\,\,{\rm and}\,\,\,\det({\rm Hess}\widetilde{\gamma})
(\theta)=0.
\] 
\end{proof}
\par 
For the symmetry set of $\widetilde{\Psi}_{\widehat{\gamma}}$, 
we have the following.\\
Let $\mathcal{W}$ be a 
$C^{\infty}$ Wulff shape and 
$\widetilde{\mathcal{W}}$ be the spherical Wulff shape of $\mathcal{W}$.
\begin{lemma}\label{lemma 4.2}
The origin ${\bf 0}$ is 
a point of symmetry set of $\partial\mathcal{W}$ if and 
only if the north polar 
$N$ is a point of spherical symmetry set of $\partial 
\widetilde{\mathcal{W}}$.
\end{lemma}
\begin {proof} 
First we give the proof of  the \lq\lq only if\rq\rq part.     
Suppose that the origin ${\bf 0}$ of $\mathbb{R}^{n+1}$ 
is a point of symmetry 
set of $\partial\mathcal{W}$. 
Then, there exist two point $x_{1}, x_{2}$ of $\partial\mathcal{W}$ such that 
$|x_{1}{\bf 0}| = |x_{2}{\bf 0}|$
and $x_{1}{\bf 0}$ (resp., $x_{2}{\bf 0}$ ) is a subset of the affine normal line 
$\ell_{1}$ (resp.  $\ell_{2}$) to  
$\partial\mathcal{W}$ at $x_{1}$(resp., $x_{2}$).
Set $\widetilde{X}=\alpha_{N}^{-1}\circ Id(X)$, where $X$ is a subset of 
$\mathbb{R}^{n+1}$. 
Since $\alpha_{N}: S^{n+1}_{N, +}\to \mathbb{R}^{n+1}\times \{1\}$ 
is the central projection at $N$, we have that the half of great circle
$\alpha_{N}^{-1}\circ Id(\ell_1)$ 
(resp., $\alpha_{N}^{-1}\circ Id(\ell_2)$) 
is the half of spherical normal to  
$\partial\widetilde{\mathcal{W}}$ at $\alpha_{N}^{-1}\circ Id(x_1)$ 
(resp., $\alpha_{N}^{-1}\circ Id(x_2)$) 
and $d(\alpha_{N}^{-1}\circ Id(x_1), N)=
d(\alpha_{N}^{-1}\circ Id(x_2), N)$.     
Therefore, if the origin is a point of 
$Sym(\partial\mathcal{W})$, then the north pole $N$ is 
a point of {\it {Spherical\mbox{-}Sym}}$(\partial\widetilde{\mathcal{W}})$. 
\par
Similarly, the \lq\lq if\rq\rq  part of Lemma \ref{lemma 4.2} follows.   
\end{proof}

\begin{remark}\label{remark 4.1}
Notice that $\gamma(\theta)=\tan(\widetilde{\gamma}(\theta))$
{\color{black}
$\gamma(\theta)={\rm tan}\left(\frac{\pi}{2}-\widetilde{\gamma}(\theta)\right)$
}
 for any 
$\theta\in S^n$. 
Since the function $\tan:(0,{\pi}/{2})\to\R_+$ is a $C^\infty$ 
diffeomorphism, 
it follows that $\theta$ is a non-degenerate critical point of  $\gamma$ 
if and only if $\theta$ is a non-degenerate critical point of 
$\widetilde{\gamma}$.
\end{remark}
\begin{proposition}\label{proposition 1}
Let $\mathcal{W}_\gamma$ be the Wulff shape 
associated with $\gamma$ and 
let $\widetilde{\mathcal{W}}_\gamma$ be the spherical Wulff shape 
associated with $\mathcal{W}_\gamma$. Then the following holds:
\begin{enumerate}
\item[(1)] {\it \mbox{Spherical-Caust}}
{\rm $\left(\partial \widetilde{\mathcal{W}}_\gamma\right)$}  
= {\it Spherical-Caust}
{\rm $\left(\partial \mathcal{D}\widetilde{\mathcal{W}}_\gamma\right)$}.
\item[(2)] {\it \mbox{Spherical-Sym}}
{\rm $\left(\partial \widetilde{\mathcal{W}}_\gamma\right)$} 
= {\it Spherical-Sym}
{\rm $\left(\partial \mathcal{D}\widetilde{\mathcal{W}}_\gamma\right)$.
}
\end{enumerate}
\end{proposition}
\begin{proof}
By Theorem \ref{theorem 1}, it follows that $\delta$ 
is of class $C^\infty$ if 
$\gamma$ is of class $C^\infty$.     
Thus, under the assumption of Theorem \ref{theorem 2}, 
the mapping $\widetilde{\Psi}_{\widehat{\delta}}: S^n\to S^{n+1}_{N,+}$ 
defined as follows is a $C^\infty$ embedding. 
\[
\widetilde{\Psi}_{\widehat{\delta}}(\theta)=
\alpha_N^{-1}\circ Id\left(\theta, \widehat{\delta}(\theta)
\right) \quad (\theta\in S^n).   
\]   
By Proposition \ref{proposition 2}, the following holds.  
\[
\widetilde{\Psi}_{\widehat{\delta}}=
D\widetilde{\Psi}_{\widehat{\gamma}}.  
\]
By the definition of spherical wave fronts, the following holds:   
\[
\left(
\widetilde{\Psi}_{\widehat{\gamma}}
\right)_t 
= 
\left(
\widetilde{\Psi}_{\widehat{\delta}}
\right)_{\pi/2-t}.    
\]
Thus, by Proposition \ref{proposition wave front}, we have the following.  
\begin{eqnarray*}
Spherical\mbox{-}Caust\!\left(\widetilde{\Psi}_{\widehat{\gamma}}\right) 
& = & 
Spherical\mbox{-}Caust\!\left(D\widetilde{\Psi}_{\widehat{\gamma}}\right)  \\ 
Spherical\mbox{-}Sym\!\left(\widetilde{\Psi}_{\widehat{\gamma}}\right) 
& = & 
Spherical\mbox{-}Sym\!\left(D\widetilde{\Psi}_{\widehat{\gamma}}\right).   
\end{eqnarray*}
Therefore, Proposition \ref{proposition 1} follows.   
\end{proof} 


\par 
\medskip
%
Now we start to prove Theorem \ref{theorem 2}.  
Firstly, recall the definition of the caustic of $\partial\mathcal{W}_\gamma$.      
\[
Caust(\partial\mathcal{W}_\gamma)
=\left\{v\;\left|\;\exists\theta\in S^n;\nabla\delta_v(\theta)=0\,\, 
\mbox{\rm and}\,\,\det(\mbox{\rm Hess}(\delta_v))(\theta)=0\right.\right\}, 
\]
where $\delta_v=\dfrac{1}{2}
\left|\left|\left(\theta, {1}/{\delta(-\theta)}\right)-v\right|\right|^2$.    
Suppose that the origin is a point of $Caust(\partial\mathcal{W}_\gamma)$. 
Then, 
there exists a point $\theta\in S^n$ such that $\nabla\delta(\theta)=0$ 
and $\det(\mbox{\rm Hess}(\delta))(\theta)=0$.
By Remark \ref{remark 4.1}, 
it follows that $\nabla\widetilde{\delta}(\theta)=0$ 
and $\det(\mbox{\rm Hess}(\widetilde{\delta}))(\theta)=0$. 
Thus, 
by Lemma \ref{lemma 4.1}, it follows that the north pole 
$N$ is contained in the spherical caustic of 
$\partial \widetilde{\mathcal{W}}_\gamma$.     
Then, by Proposition \ref{proposition 1},  
it follows that $N$ is contained in 
the spherical caustic of 
$\partial \mathcal{D}\widetilde{\mathcal{W}}_\gamma$.    
Thus, there exist $\widetilde{\theta}\in S^n$ 
such that both $\nabla\widetilde{\gamma}(\widetilde{\theta})=0$ 
and $\det(\mbox{\rm Hess}(\widetilde{\gamma}))(\widetilde{\theta})=0$ 
hold.    
Hence, again by Remark \ref{remark 4.1}, 
the origin is contained in the caustic of $\partial\mathcal{DW}_\gamma$.
\par 
\smallskip 
Next, suppose that the origin is a point of 
$Sym(\partial\mathcal{W}_\gamma)-Caust(\partial\mathcal{W}_\gamma)$.
Then, by Lemma \ref{lemma 4.2}, the north pole $N$ is 
contained in the spherical symmetry set of 
$\partial \widetilde{\mathcal{W}}_\gamma$.    
Then, by Proposition \ref{proposition 1}, 
it follows that $N$ 
is contained in the spherical symmetry set of 
$\partial \mathcal{D}\widetilde{\mathcal{W}}_\gamma$. 
Hence, again by Lemma \ref{lemma 4.2}, the origin is contained in the 
symmetry set of $\partial\mathcal{DW}_\gamma$.
\hfill $\square$
\par 
\begin{remark}\label{remark 4.2}
As a by-product of the proof of Theorem \ref{theorem 2}, 
we have the following:   
\begin{theorem}\label{theorem 4}
\begin{enumerate}
\item[(1)]\quad 
Let $\gamma: S^n\to \mathbb{R}_+$ be a $C^\infty$ {\color{black}strictly} convex integrand 
having only non-degenerate critical points.   
Then, $\delta$ is a $C^\infty$ function 
having only non-degenerate critical points.       
\item[(2)]\quad 
Let $\gamma: S^n\to \mathbb{R}_+$ be a $C^\infty$ {\color{black}strictly} convex integrand 
having only non-degenerate critical points.      
Then, the restriction of $\gamma$ 
to the set consisting critical points of 
$\gamma$ is injective if and only if 
the restriction of ${\delta}:S^n\to \mathbb{R}_+$ 
to the set consisting of critical points of 
it is injective. 
\end{enumerate}
\end{theorem}
\end{remark}
\section{Proof of Theorem \ref{theorem 3}}\label{section 5}
We first show the assertion (1) of Theorem \ref{theorem 3}.    
The following two equalities have been given 
in Subsection \ref{subsection 2.7}.   
\[
\widehat{\delta}(h_\delta(\theta))h_\delta(\theta)
=\Phi_\delta(h_\delta(\theta))
=\gamma(\theta)\theta+\nabla\gamma(\theta).
\leqno{(5.1)}
\]
\[
\widehat{\gamma}({h}_\gamma(\theta)){h}_\gamma(\theta)
=\Phi_\gamma(h_\gamma(\theta))
=
\delta(\theta)\theta+\nabla\delta(\theta).
\leqno{(5.2)}
\]
By elementary geometry, 
for any $\theta\in S^n$ the following inequalities hold.   
\[
\dfrac{1}{\delta(-h_\delta(\theta))}\geq \gamma(\theta).
\leqno{(5.3)}
\]
\[
\dfrac{1}{\gamma(-{h}_\gamma(\theta))}\geq \delta(\theta).
\leqno{(5.4)}
\]
On the other hand, since $\Phi_\delta(S^n)$ and 
$\Phi_\gamma(S^n)$ are strictly locally convex, 
for any $\theta\in S^n$ the following equivalent inequalities hold.   
\[
\gamma(\theta)\geq \frac{1}{\delta(-\theta)}.
\leqno{(5.5)}
\]
\[
\delta(\theta)\geq \frac{1}{\gamma(-\theta)}.  
\leqno{(5.6)}
\]  
The above equalities (5.1) and (5.2) imply 
that $h_\delta(\theta_0)=\theta_0$ if and only if 
$\nabla \gamma(\theta_0)=0$, 
and ${h}_\gamma(\theta_0)=\theta_0$ if and only if 
$\nabla \delta(\theta_0)=0$.
It follows that $\theta_0\in S^n$ is 
a critical point of $\gamma$ if and only if 
the following equality holds for $\theta_0\in S^n$   
\[
\dfrac{1}{\delta(-\theta_0)}= \gamma(\theta_0), 
\leqno{(5.7)}
\]
and equivalently 
that $\theta_0\in S^n$ is a critical point of $\delta$ if and only if 
the following equality holds for $\theta_0\in S^n$   
\[
\dfrac{1}{\gamma(-\theta_0)}= \delta(\theta_0).
\leqno{(5.8)}
\]
Hence, the assertion (1) of Theorem \ref{theorem 3} follows.   
\par 
\medskip
Next, we show the assertion (2) of Theorem \ref{theorem 3}.    
\par 
We first show the \lq\lq only if\rq\rq \;part.   
Let $\theta_0\in S^n$ be a non-degenerate critical point of $\gamma$ 
with Morse index $i$.    
Then, there exists 
a coordinate neighborhood $(U, \varphi)$ of $\theta_0$ such that 
$\varphi(\theta_0)={\bf 0}$ and the following equality holds: 
\[
\gamma\circ \varphi^{-1}(x_1, \ldots, x_n) = 
\gamma(\theta_0)-x_1^2-\cdots -x_i^2+x_{i+1}^2+\cdots +x_n^2.   
\leqno{(5.9)}
\]
By Theorem \ref{theorem 2} 
and the assertion (1) of Theorem \ref{theorem 3}, $-\theta_0\in S^n$ 
is a non-degenerate critical point of $\delta$.     
Thus, there exist an integer $j$ $(0\le j\le n)$ and 
a coordinate neighborhood $(V, \psi)$ 
of $-\theta_0$ such that 
$\psi(-\theta_0)={\bf 0}$ 
and the following equality holds: 
\[
\delta\circ \psi^{-1}(x_1, \ldots, x_n) = 
\delta(-\theta_0)-x_1^2-\cdots -x_j^2+x_{j+1}^2+\cdots +x_n^2.    
\leqno{(5.10)}
\]
We show that $j=n-i$.    
By (5.7), it follows that 
\[
\gamma(\theta_0)\delta(-\theta_0)=1.
\leqno{(5.11)}
\]
Set $x=(x_1, \ldots, x_n)$ and 
\begin{eqnarray*}
U_1 & = & 
\left\{x\in \varphi(U)\; \left|\; 
x_1^2+\cdots +x_i^2\ge x_{i+1}^2+\cdots +x_n^2
\right.\right\},  \\ 
U_2  & = & \left\{x\in \varphi(U)\; \left|\; 
x_1^2+\cdots +x_i^2\le x_{i+1}^2+\cdots +x_n^2
\right.\right\}, \\ 
{V}_1  & = &  \left\{x\in \psi(V)\; \left|\; 
x_1^2+\cdots +x_j^2\ge x_{j+1}^2+\cdots +x_n^2
\right.\right\}, \\ 
V_2  & = &  \left\{x\in \psi(V)\; \left|\; 
x_1^2+\cdots +x_j^2\le x_{j+1}^2+\cdots +x_n^2
\right.\right\}. 
\end{eqnarray*}
For any $x\in U_2$, by (5.9), we have the following: 
\[
\gamma\circ \varphi^{-1}(x)\ge\gamma(\theta_0).   
\]
Hence, by (5.3), (5.10) and (5.11), 
we have the following for any $x\in U_2$:
{\color{black}
\[
\frac{1}{\delta \left( -h_{\delta}\circ\phi^{-1}(x)\right)}
\geq \gamma\circ  \phi^{-1}(x)
\geq \gamma(\theta_{0})
=\frac{1}{\delta(-\theta_{0})}.
\]
This implies
}
\[
\psi\left(-h_\delta\circ\varphi^{-1}(x)\right)\in V_1.    
\]
Since $x\mapsto 
\psi\left(-h_\delta\circ\varphi^{-1}(x) 
\right)$ 
is a $C^\infty$ diffeomorphism, 
it follows that $n-i\le j$.    
Notice that (5.3) can be replaced with (5.4) 
to obtain the same inequality 
$n-i\le j$.      
On the other hand, for any $x\in U_1$, 
by (5.9), we have the following: 
\[
\gamma\circ \varphi^{-1}(x)\le \gamma(\theta_0).   
\]
Hence, by (5.5), (5.10) and (5.11), 
we have the following for any $x\in U_1$:
{\color{black}
\[
\frac{1}{\delta(-\theta_{0})}=\gamma(\theta_{0})
\geq \gamma\circ \phi^{-1}(x)
\geq \frac{1}{\delta\left( -\phi(x) \right)}.
\]
This implies
}
\[
\psi\left(-\varphi^{-1}(x)\right)\in {V}_2.    
\]
Since $x\mapsto \psi\left(-\varphi^{-1}(x)
\right)$ 
is a $C^\infty$ diffeomorphism, 
it follows that $i\le n-j$.  
Therefore, we have $j=n-i$.  
\par 
The \lq\lq if\rq\rq\;part can be proved by the same method.       
Thus, the assertion (2) of Theorem \ref{theorem 3} follows.    
\hfill $\square$
\section*{Acknowledgement}
{\color{black} This work is partially supported 
by JSPS and CAPES 
under the Japan--Brazil research cooperative program and 
JSPS KAKENHI Grant Number 26610035, {\color{black}17K05245}.}

\end{document}